\documentclass[12pt,reqno,draft]{amsart}
\usepackage[top=1in, bottom=1in, left=1in, right=1in]{geometry}
\usepackage{times, amsthm, amssymb, amsmath, amsfonts, amsthm, bm, graphicx, mathrsfs}
\usepackage{latexsym,wasysym}
\usepackage{amscd}
\usepackage{float}
\usepackage{enumerate}
\usepackage[all]{xy}
\usepackage[usenames,dvipsnames]{color}
\usepackage{tikz}
\usepackage{graphicx}
\graphicspath{ {electronics/} }
\usepackage[pagebackref=true]{hyperref}
\usepackage{lineno}

\parskip=\medskipamount

\usepackage{algpseudocode}

\usepackage{xy}
\input xy
\xyoption{all}
\newcommand{\excise}[1]{}%{$\star$\textsc{#1}$\star$}

\newtheorem{theorem}{Theorem}[section]
\newtheorem{proposition}[theorem]{Proposition}
\newtheorem{lemma}[theorem]{Lemma}
\newtheorem{corollary}[theorem]{Corollary}

\theoremstyle{definition}
\newtheorem{definition}[theorem]{Definition}
\newtheorem{example}[theorem]{Example}
\newtheorem{remark}[theorem]{Remark}

\newcommand\set[1]{\left\{#1\right\}}
\newcommand\abs[1]{|#1|}

\begin{document}%%%%%%%%%%%%%%%%%%%%%%%%%%%%%%%%%%%%%%%%%%%%%%%%%%%%%%%%
%%%%%%%%%%%%%%%%%%%%%%%%%%%%%%%%%%%%%%%%%%%%%%%%%%%%%%%%%%%%%%%%%%%%%%%%

\mbox{}
\title{Expansion of permutations as products of transpositions}
\author{Michael Anshelevich, Matthew Gaikema, Madeline Hansalik, Songyu He, Nathan Mehlhop}
\address{Mathematics Department\\Texas A\&M University\\College Station, TX 77843-3368}
\email{manshel@math.tamu.edu, matt.gaikema@tamu.edu, mahans95@tamu.edu, he1007@tamu.edu, mehl144@tamu.edu}
\date{\today}

\begin{abstract}
We compute the number of ways a given permutation can be written as a product of exactly $k$ transpositions. We express this number as a linear combination of explicit geometric sequences, with coefficients which can be computed in many particular cases. Along the way we prove several symmetry properties for matrices associated with bipartite graphs, as well as some general (likely known) properties of Young diagrams. The methods involve linear algebra, enumeration of border strip tableau, and a differential operator on symmetric polynomials.
\end{abstract}
\maketitle

% \setcounter{tocdepth}{1}
% \tableofcontents

%%%%%%%%%%%%%%%%%%%%%%%%%%%%%%%%%%%%%%%%%%%%%%%%%%%%%%%%%%%%%%%%%%%%%%%%%
\section{Introduction}\label{s:intro}%%%%%%%%%%%%%%%%%%%%%%%%%%%%%%%%%%%%
%raggedbottom%%%%%%%%%%%%%%%%%%%%%%%%%%%%%%%%%%%%%%%%%%%%%%%%%%%%%%%%%%%%

Let $\alpha \in S_n$ be a permutation of $n$ objects. We would like to compute the number $c_k(\alpha)$ of ways of writing $\alpha$ as a product of exactly $k$ transpositions,
\[
c_k(\alpha) = \abs{\set{(\pi_1, \pi_2, \ldots, \pi_k): \alpha = \pi_1 \pi_2 \ldots \pi_k, \pi_i\text{'s transpositions}}}.
\]
It is easy to see (Lemma~\ref{Lemma:c_k}) that $c_k(\alpha)$ in fact depends only on the cycle type (conjugacy class) of $\alpha$, and so should be considered as a function not of a permutation but of an (integer) partition $\mu$. This type of question has been considered in the literature. For example, Jackson \cite[Theorem 3.1]{Jackson-Product-conjugacy} proves a more general result. However, we will see that as a function of $k$, $c_k(\mu)$ is a linear combination of a number of geometric sequences, with certain symmetries between coefficients, and such qualitative properties of $c_k$ are far from clear from Jackson's formula. Goulden \cite{Goulden-Differential} proved an explicit formula in the case that $\alpha$ is a single cycle,
\begin{equation}
\label{Goulden}
c_k(n) = \frac{1}{n!} \sum_{i=0}^{n-1} \binom{n-1}{i} (-1)^i \left[ \binom{n}{2} - n i \right]^k
\end{equation}
(see also Corollary 5.3 in \cite{Levy-Schur-Weyl}). This formula served as a starting point of our investigations, and we are grateful to Richard Stanley for bringing it to our attention.

Besides intrinsic interest, the motivation for this investigation comes from the study of matrix-valued Brownian motions, in the style of \cite{Levy-Schur-Weyl}. This direction will be pursued elsewhere.

The numbers $c_k(\mu)$ can be computed from powers of the following matrix $A_n$.

\begin{definition}
\label{Definition:A_n}
We define $A = A_n$ to be the square matrix with rows and columns indexed by partitions $\lambda \in P(n)$ of $n$, in lexicographic order, with each entry $A_{\lambda \sigma}$ giving the number of single transpositions which take a permutation of cycle type $\lambda$ to a permutation of cycle type $\sigma$; these entries are given explicitly in Proposition~\ref{Proposition:Transition-matrix}.
\end{definition}

Thus $(A_n^k)_{\lambda \sigma}$ represents the number of ways $k$ transpositions can take $\lambda$ to $\sigma$. It follows that $c_k(\mu)$ is obtained by starting with the vector corresponding to the identity permutation, applying $A_n$ $k$ times, and finding the entry corresponding to $\mu$ in the resulting vector.

\begin{example}
Let $n=4$, so that the corresponding partitions are $1111 \prec 211 \prec 22 \prec 31 \prec 4$. Then using Proposition~\ref{Proposition:Transition-matrix} below,
\[
A_4=\left(
\begin{array}{ccccc}
 0 & 6 & 0 & 0 & 0 \\
 1 & 0 & 1 & 4 & 0 \\
 0 & 2 & 0 & 0 & 4 \\
 0 & 3 & 0 & 0 & 3 \\
 0 & 0 & 2 & 4 & 0 \\
\end{array}
\right),
\]
and
\[
\left(\begin{array}{c}
    120 \\
    0 \\
    104 \\
    108 \\
    0
    \end{array}\right) =
A_4^4\left(\begin{array}{c}
    1 \\
    0 \\
    0 \\
    0 \\
    0
    \end{array}\right).
\]
By looking at the appropriate entry of this vector, we learn for example that for $\mu = 31$, $c_{4}(31) = 108$.
\end{example}

Our approach for computing the numbers $c_k(\mu)$ is to first diagonalize the transition matrix $A_n$. From Matlab calculations in Example~\ref{Example:Matlab}, we observed the following general properties of the matrix $A_n$.

\begin{proposition}
\label{Proposition:A-symmetries}
Let $A_n$ be the transition matrix just described.
\begin{enumerate}
\item
For $n \geq 2$, $A_n$ has a zero eigenvalue.
\item
The remaining eigenvalues of $A_n$ come in pairs: if $\rho$ is an eigenvalue, so is $-\rho$.
\item
The eigenvectors of $A_n$ can be chosen so that if $(v(\nu))_{\nu \in P(n)}$ is an eigenvector with eigenvalue $\rho$, then $((-1)^{n - \ell(\nu)} v(\nu))_{\nu \in P(n)}$ is an eigenvector with eigenvalue $-\rho$.
\end{enumerate}
\end{proposition}

Since four of the authors of this paper are undergraduate students, who worked on different aspects of the project, we obtained several proofs of the preceding proposition. Some of them are done in the general context of weighted adjacency matrices of graphs, while others follow from the explicit combinatorial description of the structure of the matrix $A_n$ in Proposition~\ref{Proposition:Eigenvectors-A} below.

Next, we explicitly diagonalize the matrix $A_n$. In the following proposition the coefficients $\chi^\lambda(\nu)$ arise from expanding Schur polynomials in the power sum symmetric polynomials; they can also be interpreted as character values (an approach we do not use), and have explicit combinatorial expressions $\chi^\lambda(\nu) = \sum_T(-1)^{ht(T)}$ in terms of border strip tableaux $T$ (of shape $\lambda$ filled with $\nu$), see Section~\ref{Section:Border-strip}.

\begin{proposition}
\label{Proposition:Eigenvectors-A}
For each $\lambda$, the vector $(\chi^\lambda(\nu))_{\nu \in P(n)}$ is an eigenvector of the matrix $A_n$ with eigenvalue
\begin{equation}
\label{Equation:Rho}
2 \rho_\lambda = \sum_i \lambda_i (\lambda_i - 2 i + 1).
\end{equation}
\end{proposition}

As a consequence, we obtain our main result.

\begin{theorem}
\label{theorem;Main}
Let $\mu \in P(n)$ be a partition. Then
\begin{enumerate}
\item
\[
c_k(\mu) = \frac{1}{n!} \sum_{\lambda \in P(n)} \chi_{\lambda, \mu} \rho_\lambda^k,
\]
where $\chi_{\lambda, \mu} = \chi^\lambda(1^n) \chi^\lambda(\mu)$ are certain integer coefficients:  $\chi^\lambda(1^n)$ is the number of standard Young tableaux of shape $\lambda$, and $\chi^\lambda(\mu)$ can be computed in terms of border strip tableaux.
\item
$\rho_{\lambda'} = - \rho_\lambda$, where $\lambda'$ is the conjugate partition to $\lambda$; in particular, if $\lambda$ is self-conjugate then $\rho_\lambda = 0$.
\item
$\chi_{\lambda', \mu} = (-1)^{n - \ell(\mu)} \chi_{\lambda, \mu}$.
\item
If
\[
f_\mu(z) = \sum_{k=0}^\infty \frac{1}{k!} c_k(\mu) z^k
\]
is the generating function of the $c_k$'s, then
\[
\begin{split}
f_\mu(z)
& = \sum_{\lambda \in P(n)} \chi_{\lambda, \mu} \exp(\rho_\lambda z) \\
& =
\begin{cases}
\sum_{\lambda \in P(n)} \chi_{\lambda, \mu} \cosh(\rho_\lambda z) = \sum_{\lambda \prec \lambda'} 2 \chi_{\lambda, \mu} \cosh(\rho_\lambda z), & n - \ell(\lambda) \text{ even}, \\
\sum_{\lambda \in P(n)} \chi_{\lambda, \mu} \sinh(\rho_\lambda z) = \sum_{\lambda \prec \lambda'} 2 \chi_{\lambda, \mu} \sinh(\rho_\lambda z), & n - \ell(\lambda) \text{ odd}.
\end{cases}
\end{split}
\]
where $\prec$ is the lexicographic ordering.
\end{enumerate}
\end{theorem}

\begin{example}
Let $\mu = n$, so that the corresponding permutation consists of a single cycle. Then using the ideas from Section~\ref{Section:Border-strip},
we see that $\chi^\lambda(\mu)$ is the sum over border strip tableaux of shape $\lambda$ all of whose entries are the same. Thus, it is non-zero only if $\lambda$ itself is a border strip, and so a hook, in which case $\chi^\lambda(\mu) = (-1)^{\ell(\lambda) - 1}$. In this case also $\chi^\lambda(1^n)$, the number of standard Young tableaux of shape $\lambda$, is easily seen to be $\binom{n - 1}{\ell(\lambda) - 1}$. Finally, for a hook $\lambda$,
\[
\rho_\lambda = \frac{1}{2} (n - \ell(\lambda) + 1)(n - \ell(\lambda)) + \sum_{i=2}^{\ell(\lambda)} (1 - i)
= \frac{1}{2} (n^2 - n) - n (\ell(\lambda) - 1).
\]
Thus we recover formula~\eqref{Goulden}.
\end{example}

Further explicit formulas are obtained at the very end of the article.

\textbf{Acknowledgements.} We are grateful to Richard Stanley for pointing out to us reference \cite{Goulden-Differential}, which served as the starting point for our investigations. We would also like to thank Amudhan Krishnaswamy-Usha and Laura Matusevich for helpful comments, and Michael Brannan and Ken Dykema for their support during the project.

%%%%%%%%%%%%%%%%%%%%%%%%%%%%%%%%%%%%%%%%%%%%%%%%%%%%%%%%%%%%%%%%%%%%%%%%%

\section{Background}

%\subsection{Transpositions}

%\subsection{Conjugacy Classes}
%\begin{lemma}
%Conjugacy classes in a group $G$ form a partition of $G$.
%\end{lemma}

%\begin{proof}
%% http://www.math.wm.edu/~vinroot/415conj.pdf
%To show this we need to prove that conjugacy is an equivalence relation.
%For $g,h\in G$, $g$ is conjugate to $h$, or $g\sim h$ iff $\exists x\in G:xgx^{-1}=h$.
%This is a reflexive relation, since for any $g\in G$ we have $ege^{-1}=g$,
%where $e$ is the identity element.
%Now say $g\sim h$, so $\exists x\in G: xgx^{-1}=h$.
%This implies that $g=x^{-1}hx$, so the symetric property holds.
%Now suppose that for $g,h,a\in G$, $g\sim h$ and $h\sim a$,
%so $xgx^{-1}=h$ and $yhy^{-1}=a$.
%This gives us $yxgx^{-1}y^{-1}=yxg(yx)^{-1}=a$, so the transitive property holds.
%Thus we have an equivalence relation.
%\end{proof}

For a partition $\lambda \in P(n)$, we denote by $\ell(\lambda)$ its number of parts. We call the partition \emph{even} if it has an even number of parts, and \emph{odd} if it has an odd number of parts.

For a permutation $\alpha \in S_n$ and a partition $\lambda \in P(n)$, we say that $\alpha$ has \emph{cycle type $\lambda$} if the sizes of cycles in $\alpha$ are precisely the parts of $\lambda$. Recall that if we follow \cite{Stanley-volume-2} and define, for a partition
\[
\lambda = 1^{k_1} 2^{k_2} \ldots,
\]
\[
z_\lambda = 1^{k_1} 1! 2^{k_2} 2! \ldots,
\]
then the number of permutations of cycle type $\lambda$ is $\frac{n!}{z_\lambda}$.

\begin{lemma}
Two permutations lie in the same conjugacy class of $S_n$ if and only if they have the same cycle type.
\end{lemma}

\begin{proof}
It suffices to note that if $\alpha$ has the cycle structure
\[
\alpha = \Bigl(u_1(1) \ldots u_1(\lambda_1) \Bigr) \ \Bigl(u_2(1) \ldots u_2(\lambda_2) \Bigr) \ \ldots \ \Bigl(u_j(1) \ldots u_j(\lambda_j) \Bigr)
\]
and $\gamma$ is any permutation, then
\[
\gamma \alpha \gamma^{-1} = \Bigl(\gamma(u_1(1)) \ldots \gamma(u_1(\lambda_1)) \Bigr) \ \Bigl(\gamma(u_2(1)) \ldots \gamma(u_2(\lambda_2)) \Bigr) \ \ldots \ \Bigl(\gamma(u_j(1)) \ldots \gamma(u_j(\lambda_j))\Bigr). \qedhere
\]
\end{proof}

\begin{lemma}
\label{Lemma:c_k}
If $\alpha$ and $\beta$ lie in the same conjugacy class of $S_n$, then $c_k(\alpha) = c_k(\beta)$. Thus we may, and will, consider $c_k$ as a function of a partition $\mu \in P(n)$.
\end{lemma}

\begin{proof}
If $\alpha = \gamma \beta \gamma^{-1}$, then
\[
\beta = \pi_1 \pi_2 \ldots \pi_k
\]
is equivalent to
\[
\alpha = (\gamma \pi_1 \gamma^{-1}) (\gamma \pi_2 \gamma^{-1}) \ldots (\gamma \pi_k \gamma^{-1}),
\]
and each $\gamma \pi_i \gamma^{-1}$ has the same cycle type as $\pi_i$.
\end{proof}

\begin{lemma}
\label{Lemma:Split}
Let $\pi\in S_n$ be a transposition,
and $\alpha\in S_n$ be a permutation.
Take $\pi\alpha$.
If $\pi = (i,j)$, and $i$ and $j$ are both in the same cycle of $\alpha$, this cycle gets cut into two cycles.
If $i,j$ are in different cycles, they get glued together into a single cycle.
\end{lemma}

\begin{proof}
Say $\pi=(i,j)$ and
\[
\alpha=(i,\alpha_{1,2},\dots,\alpha_{1,x})(j,\alpha_{2,2},\dots,\alpha_{2,y}) \ldots (\ldots).
\]
Then $\pi\alpha(\alpha_{1,x})=\pi\big(\alpha(\alpha_{1,x})\big)=\pi(i)=j$
and $\pi\alpha(\alpha_{2,y})=\pi(j)=i$.
Therefore,
\[\pi\alpha=(i,\alpha_{1,2},\dots,\alpha_{1,x},j,\alpha_{2,2},\dots,\alpha_{2,y}) \ldots (\ldots),\]
and we see the two cycles are glued together.

Now say $\pi=(i,j)$ and $\alpha=(i, \alpha_2, \ldots, \alpha_x, j, \alpha_{x+2},\dots,\alpha_y) \ldots (\ldots)$.
Then $\pi\alpha(i)= \pi(\alpha_2) = \alpha_2$, $\pi \alpha (\alpha_x) = \pi(j) = i$, $\pi \alpha(j) = \pi(\alpha_{x+2}) = \alpha_{x+2}$, and $\pi \alpha (\alpha_y) = \pi(i) = j$. Thus \[\pi\alpha=(i,\alpha_2, \ldots, \alpha_x)(j,\alpha_{x+2},\ldots,\alpha_y) \]
and the single cycle is split into two, their sized determined by the relative position of $i$ and $j$ in the cycle of $\alpha$.
\end{proof}

%\subsection{Graphs}

%%%%%%%%%%%%%%%%%%%%%%%%%%%%%%%%%%%%%%%%%%%%%%%%%%%%%%%%%%%%%%%%%%%%%%%%%
%%%%%%%%%%%%%%
%\section{Reformulation of the problem in terms of graphs and matrices}

\section{Main results}

\subsection{Transposition formulas and Integer partition Matrix}

\begin{proposition}
\label{Proposition:Transition-matrix}
Note that for a permutation $\alpha$ with cycle type $\lambda$, the number of transposition taking $\alpha$ to a permutation of cycle type $\sigma$ depends only on $\lambda$ and not on $\alpha$. The entries $A_{\lambda \sigma}$ of the transition matrix from Definition~\ref{Definition:A_n} are zero unless $\lambda$ and $\sigma$ coincide, except one term in $\lambda$ is split into two in $\sigma$, or vice versa. More precisely, let
\[
\lambda = 1^{k_1} 2^{k_2} \ldots n^{k_n}.
\]
Then we have four cases when the entry $A_{\lambda \sigma} \neq 0$:
\begin{enumerate}
\item[Case 1.] We are splitting a cycle of size $m$ into two cycles of different lengths $i \neq j$, $i + j = m$.
\item[Case 2.] We are splitting a cycle of size $m$ into two cycles of the same length $i$, $2i = m$.
\item[Case 3.] We are combining two cycles of the same length $i$ into one cycle of size $m = 2i$.
\item[Case 4.] We are combining two cycles of different lengths $i$ and $j$ into one cycle of size $m = i + j$.
\end{enumerate}
Then the entries of the matrix $A_n$ corresponding to these four cases are
\begin{center}
\begin{tabular}{ |c|c| }
 \hline
 Case 1 & $ij(k_i+1)(k_j +1)$.  \\
 Case 2 & $\frac{i^2(k_i+1)(k_i+2)}{2}$  \\
 Case 3 & $\frac{m(k_m+1)}{2} = i(k_m+1)$\\
 Case 4 & $(i+j)(k_m+1)$ \\
 \hline
\end{tabular}
\end{center}
\end{proposition}

\begin{proof}
If we have integer $n$, the matrix $A_n$, is a square matrix with dimension of the number of partitions of $n$. Each column is labeled with a partition, and each row as well, where the ordering of the partitions is (reverse) lexicographical. Each entry in the matrix represents the number of transpositions that will take the partition assigned to that column, to the partition assigned to that row.  According to Lemma~\ref{Lemma:Split}, this number is non-zero precisely in the four cases from the statement of the proposition.

We know we can calculate this number of transpositions, by taking the number of transpositions that take a given column partition to the row partition, and scaling this number by the number of permutations of the column type divided by the number of permutations of the row type. We derive four formulae to calculate these matrix entries for each of four cases.

In case 1 where we are splitting a cycle of length $m$ into two cycles of different lengths $i$ and $j$, we first find the number of possible permutations of the column type and the row type to be, respectively,

$$ \frac{n!}{1^{k_1}2^{k_2}...n^{k_n}k_1!k_2!...k_n!}$$

and

$$ \frac{n!}{...i^{k_1+2}j^{k_j+1}m^{k_m-1}(k_j+2)!(k_j+1)!...(k_m-1)!}.$$

Taking the ratio of these two quantities gives us our scaling factor.
$$\frac{ij(k_i+1)(k_j +1)m}{m k_m}.$$
Now we multiply by the number of transpositions that take the column type to the row type. This multiplying by $k_m$ gives us that the number of total transpositions is equal to $ij(k_i+1)(k_j +1)$.

In case 2 where we are splitting a cycle of length $m$, but instead with $i=j$ our formula only changes slightly to

$$\frac{k_mi^2(k_i+1)(k_i+2)}{m k_m}\frac{m}{2} = \frac{i^2(k_i+1)(k_i+2)}{2}.$$

For case 3 and 4 we want to know the number of transpositions that will glue two cycles of length $i$ and $j$ into one cycle of length $m$. We derive these formulae in the same manner. Our scaling factor (ratio of number of possible permutations of each cycle type), is derived by dividing

$$ \frac{n!}{1^{k_1}2^{k_2}...n^{k_n}k_1!k_2!...k_n!}$$

by $$\frac{n!}{i^{k_i-1}j^{k_j-1}m^{k_m+1}(k_i-1)!(k_j-1)!...(k_m+1)!}.$$

If $i$ does not equal $j$,
$$(i+j)(k_m+1).$$
If $i=j$,
$$ \frac{m(k_m+1)}{2} = i(k_m+1).$$
The result follows.
\end{proof}

\begin{example}
\label{Example:Matlab}
In Matlab, we wrote a program which will construct this matrix $A$, using the four formulas above. Upon diagonalizing the resulting matrix $A$, we find the eigenvalues of $A$ for each $n$. The following table gives the eigenvalues for the matrix $A$ for $n=3$ through $n=10$.

\begin{tabular}{r|rrrrrrrrrrrrrrrrr}
3&-3&0&3&&&&&&&&&&&&&& \\
4&-6&6&-2&2&0&&&&&&&&&&&& \\
5&-10&10&-5&5&-2&0&2&&&&&&&&&& \\
6&-15&15&-9&9&-5&5&0&-3&-3&3&3&&&&& \\
7&-21&21&-14&14&-9&9&-7&-6&7&6&-3&3&-1&1&0&& \\
8&-28&28&-20&20&-14&-12&14&-10&-8&-7&12&10&8&7&-2&2&0 \\
&-4&-4&4&4&&&&&&&&&&&&& \\
9&-36&36&-27&27&-20&-18&20&18&-15&15&-9&-8&9&8&-4&-3&4 \\
&3&-1&1&-12&-12&12&12&-6&-6&6&6&0&0&&&& \\
10&45&-45&35&-35&27&25&-27&-25&21&-21&18&17&-18&-17&13&11&10 \\
&9&7&-13&-11&-10&-9&-15&15&15&-7&-15&5&0&0&5&5&3 \\
&3&3&-5&-5&-5&-3&-3&-3
\end{tabular}

%$\begin{tabular}{l*{6}{c}r}
%              3 & 4 & 5  & 6 & 7 & 8 & 9 & 10\\
%\hline
%    -3 &-6 & -10&-15&-21&-28&-36&45\\
%    0 &	6&	10&	15&	21&	28&	36&	-45  \\
%    3 &	-2	&-5&	-9&	-14&	-20&	-27&	35\\
%     &2&	5&	9&	14&	20&	27&	-35\\
%     &0	&-2	&-5&	-9&	-14&	-20&	27\\
%     &&0&	5	&9	&-12&	-18&	25 \\
%     &&2&	0	&-7&	14&	20&	-27\\
%     &&&-3&	-6&	-10&	18&	-25\\
%&&&-3	&7&	-8&	-15&	21\\
%&&&3&	6	&-7&	15&	-21\\
%&&&3&	-3&	12&	-9	&18\\
%&&&&3&	10&	-8&	17\\
%&&&&-1&	8&	9&	-18\\
%&&&&1&	7	&8&	-17\\
%&&&&0	&-2	&-4	&13\\
%&&&&&2	&-3&	11\\
%&&&&&0	&4&	10\\
%&&&&&-4	&3&	9\\
%&&&&&-4	&-1&	7\\
%&&&&&4	&1	&-13\\
%&&&&&4	&-12&	-11\\
%&&&&&0	&-12	&-10\\
%&&&&&&12&	-9\\
%&&&&&&12&	-15\\
%&&&&&& -6&	15\\
%&&&&&& -6&	15\\
%&&&&&& 6	&-7\\
%&&&&&& 6	&-15\\
%&&&&&& 0	&5\\
%&&&&&& 0	&0\\
%&&&&&&&0\\
%&&&&&&&5\\
%&&&&&&&5\\
%&&&&&&&3\\
%&&&&&&&3\\
%&&&&&&&3\\
%&&&&&&&-5\\
%&&&&&&&-5\\
%&&&&&&&-5\\
%&&&&&&&-3\\
%&&&&&&&-3\\
%&&&&&&&-3\\
%
%
%    \end{tabular}$

It should be noted that aside from the eigenvalue of zero, each eigenvalue comes as a pair with its negative. Furthermore, each set of eigenvalues contains at least one which is zero. Each set of eigenvalues is also bounded above and below by $n$ choose $2$. We also computed the corresponding eigenvectors, and observed their properties described in Proposition~\ref{Proposition:A-symmetries}, but do not include them here for brevity.
\end{example}

%%%%%%%%%%%%%%%%%%%%%%%%%%%%%%%%%%%%%%%%%%%%%%%%%%%%%%%%%%%%%%%%%%%%%%%%%%%

\subsection{Properties of the transition matrix}

Denote also by $t_{\lambda \mu}$ the number of transpositions taking a permutation of cycle type $\lambda$ to permutations of cycle type $\mu$.

%Proposition 8.2\\
\begin{proposition}
The matrix $A_n$, already seen to have entries $t_{\lambda \mu} \frac{z_\mu}{z_\lambda}$, is the transpose of the matrix with entries $t_{\lambda \mu}$. Consequently the columns of $A_n$ add up to $\binom{n}{2}$.\\
\end{proposition}

\begin{proof}
The sum of the number of transpositions that take a given permutation $\tau$ to each other permutation is exactly $\binom{n}{2}$. This is because we can construct exactly $\binom{n}{2}$ transpositions, and each of these will take $\tau$ to another permutation. \\
We would like to understand why the number of transpositions that take each permutation $\lambda$ to a given permutation $\tau$, also sum to $\binom{n}{2}$ after each is scaled by the total number of permutations of cycle type $\lambda$ divided by the total number of permutations of cycle type $\tau$. \\
The claim therefore is, that
$$ {\text{\# trans from type }\tau \text{ to type }\lambda} = {\text{\# trans from type }\lambda \text{ to type } \tau}\bigg( \frac{\text{\# perm of type }\lambda}{ \text{\# perm of type }\tau}\bigg)$$
And when we sum over all $\lambda$, both sides are equal to $\binom{n}{2}$.\\
Therefore we must prove that $$ \frac{\text{\# permutations of type }\lambda}{ \text{\# permutations of type }\tau} = \frac{\text{\# transpositions from type }\tau \text{ to type }\lambda}{\text{\# transpositions from type }\lambda \text{ to type } \tau}$$

Suppose we have an array $L$ of all permutations of type $\lambda$, and $T$ be an array of all permutations of type $\tau$.
$$ L = \{\lambda_1, \lambda_2, \lambda_3, .... \lambda_k\}$$
$$T = \{\tau_1, \tau_2, \tau_3, .... \tau_s\}$$
An arbitrary $\lambda_p$ is mapped to a set number of elements, $a$, of $T$ when acted upon by all possible transpositions. Since every $\lambda$ is a permutation of the same type, each $\lambda_p$ has the same number of transpositions that map it to elements of $T$.
These $a$ elements are mapped back to our particular $\lambda_p$ by the same transpositions. In other words, for every transposition that takes an element of $L$ to an element of $T$, that same transposition takes an element of $T$ to an element of $L$. And likewise, every transposition that takes an element of $T$ to an element of $L$ also takes an element of $T$ back to an element of $L$. Therefore the number of transpositions of type $\lambda$, $k$, multiplied by $a$, must be equal to the number of transpositions which take any element of $T$ to elements of $L$, call it $b$, multiplied by the number of elements of $T$. That is,
$$ a*k = b*s$$
and so
$$\frac{k}{s} = \frac{b}{a}.$$
Therefore we have as we want, that:
$$ \frac{\text{\# permutations of type }\lambda}{ \text{\# permutations of type }\tau} = \frac{\text{\# transpositions from type }\tau \text{ to type }\lambda}{\text{\# transpositions from type }\lambda \text{ to type } \tau}$$
\end{proof}

\begin{example}
EXAMPLE\\
Let us look at the integer partitions of $4$. We have five partitions:\\
1 1 1 1 \\
2 1 1\\
2 2\\
3 1\\
4\\
The number of transpositions from 2 1 1 to 1 1 1 1 is 1.\\
The number of transpositions from 2 1 1 to 3 1 is 4.\\
The number of transpositions from 2 1 1 to 2 2 is 1.\\
The number of transpositions from 2 1 1 to 4 is 0.\\
This adds up to 6, which is $\binom{n}{2}$ which we expect and understand.\\

On the other hand,\\
The number of transpositions from 1 1 1 1 to 2 1 1 is 6.\\
The number of transpositions from 3 1 to 2 1 1 is 3. \\
 The number of transpositions from 2 2 to 2 1 1 is 2.\\
 The number of transpositions from 4 to 2 1 1 is 0.\\
 These do not add up to $\binom{n}{2} = 6$, but once we scale each of these numbers by the number of possible permutations of each type divided by the number of permutations of type 2 1 1, we do again get 6:

$$ 6 \bigg( \frac{1}{6}\bigg) + 3\bigg(\frac{8}{6}\bigg) + 2*\bigg(\frac{3}{6}\bigg) = 1 + 4 + 1 = 6$$
\\
\end{example}

\section{Main results}

\subsection{Graph theory techniques}

\begin{definition}
A \emph{simple weighted directed graph} is an object $\Gamma = (V, E, f)$. Here $V$ is a set of vertices,
\[
E \subset \set{(a,b): a, b \in V}
\]
is a set of edges, the pair $(a,b)$ being thought of as the edge from $a$ to $b$, and
\[
f : E \rightarrow \mathbb{N}
\]
assigns to each edge an (integer) weight. To each such graph there corresponds a matrix $A_\Gamma$, which is square with row and columns indexed by $V$, and entries $A_{ab} = f((a,b))$ if $(a,b) \in E$ and zero otherwise.

A graph is \emph{bipartite} if we can decompose $V = V_1 \sqcup V_2$ so that the only edges $(a,b)$ satisfy $a \in V_1, b \in V_2$ or vice versa. Equivalently, the matrix of a bipartite graph can be decomposed as
\[
A =
\begin{pmatrix}
0 & X \\
Y & 0
\end{pmatrix},
\]
where $X, Y$ are rectangular matrices.

In particular, we call the \emph{partition graph} the graph with vertices $V = P(n)$, the set of partitions of $n$, and $(\lambda, \mu) \in E$ if there are permutations $\alpha$ of cycle type $\lambda$ and $\beta$ of cycle type $\mu$ such that $\beta = \lambda \tau$, for $\tau$ a transposition. We define the weight function on the partition graph as
\[
f((\lambda, \mu)) = t_{\lambda \mu} \frac{z_\mu}{z_\lambda}.
\]
As already observed, the matrix corresponding to this graph is precisely $A_n$.
\end{definition}

\begin{proposition}
\label{Prop:bipartite-zero}
Let $\Gamma$ be a graph with matrix $A_\Gamma$.

\begin{enumerate}
\item
Suppose that for every permutation $\sigma$ of $V$, there is a vertex $a \in V$ such that the pair $(a, \sigma(a))$ is not an edge. Then $A_\Gamma$ has a zero eigenvalue.
\item
Suppose $\Gamma$ is bipartite, with parts $V_1$ and $V_2$. If $\abs{V_1} \neq \abs{V_2}$, then  $A_\Gamma$ has a zero eigenvalue.
\end{enumerate}
\end{proposition}

\begin{proof}
$(a)$ Suppose not, for every permutation $\sigma$ of $V$, there is a vertex $a \in V$ such that the pair $(a, \sigma(a))$ is not an edge, but $det(A_{\Gamma}) \neq 0$.

Recall that $det(A_{\Gamma}) = \sum\limits_{i_1i_2...i_n \in S_n}a_{i_1,1}a_{i_2,2}a_{i_3,3} \dots a_{i_n,n}$ where $i_1i_2...i_n$ is a permutation of $123...n$ and $S_n$ is the collection of all permutations of objects $123...n$. Here, we assume $|V| = n$ and $123...n$ is a projection of each vertices in V.

Since $det(A_{\Gamma}) \neq 0$, then $\exists j_1j_2j_3...j_n \in S_n \text{ such that }a_{j_1,1}a_{j_2,2}a_{j_3,3} \dots a_{j_n,n} \neq 0 $. Therefore, $a_{j_i,i} \neq 0 \text{ for all } i \leq n.$  But since $j_1j_2j_3...j_n$ is a permutation of $123...n$, thus we can define $\sigma$ such that $\sigma(123...n) = j_1j_2j_3...j_n$. Since all $a_{j_i,i} \neq 0$, by definition of graph, the pair$(j_i,\sigma(j_i)) = (j_i,i)$ is an edge for all $j_i \in V$ which contradicts to the hypothesis. Therefore, $(a)$ is true.\\

$(b)$ By the definition of a bipartite graph, we can transfer $A_\Gamma$ to another matrix $B_\Gamma =
\begin{pmatrix}
0 & X \\
Y & 0
\end{pmatrix}$ by swapping rows and columns, and $|det(A_\Gamma)|=|det(B_\Gamma)|$.

Notice, $det(B_{\Gamma}) = \sum\limits_{i_1i_2...i_n \in S_n}b_{i_1,1}b_{i_2,2}b_{i_3,3} \dots b_{i_n,n}$, where $i_1i_2...i_n$ is a permutation of $123...n$ and $S_n$ is the collection of all permutations of objects $123...n$.. Without losing generality, we can assume $V_1 = \{1,2,3,..,j \}$, where $0 < n-j < j \Rightarrow |V_2| < |V_1|$.

$\forall \sigma(123...n) \in S_n$, define the set $\sigma (V_1) = \{ \sigma(i)| i \in V_1\}$. Since $\sigma$ is a permutation, we have $|\sigma(V_1)|=|V_1|$. But $|V_2| < |V_1| \Rightarrow  |V_2| < |\sigma(V_1)|$. Therefore, $ \exists \sigma (l) \in \sigma(V_1)$ such that $\sigma (l) \notin V_2$.

By the definition of the bipartite graph, $b_{\sigma (l),l} = 0$. Since $\sigma$ and $ l$ are arbitrary, we have $det(B_{\Gamma}) = \sum\limits_{i_1i_2...i_n \in S_n}b_{i_1,1}b_{i_2,2}b_{i_3,3}\dots b_{i_l,l}\dots  b_{i_n,n} = \sum\limits_{S_n}0=0$. Therefore $det(A_\Gamma)=det(B_\Gamma)=0$, $A_\Gamma$ has 0 eigenvalue follows.
\end{proof}

The following result can be proved by generating function methods from Section~3.4 of \cite{Aigner-Course-Enumeration}. Its first appearance seems to be identity XVI in \cite{Glaisher} (as combined with Euler's theorem that the number of partitions into odd distinct parts equals the number of self-conjugate partitions).

\begin{proposition}
\label{Prop:Aigner}
For $n > 2$, the number of even partitions of $n$ is different from the number of the odd partitions of $n$. In fact, the difference between these quantities is, up to a sign, the number of self-conjugate partitions of $n$.
\end{proposition}

\begin{proof}[First proof of Proposition~\ref{Proposition:A-symmetries} part (a)]

We approached this problem by using the graph theory.

First, to arrange all integer partitions of a number in such a way that all different integer partitions which have the same number of cycles (e.g. for $n=4, 2+2=4$ and $3+1=4$, and both 2,2 and 3,1 have two cycles, thus 22 and 31 are in the same level). Here is an example of integer partition of n:
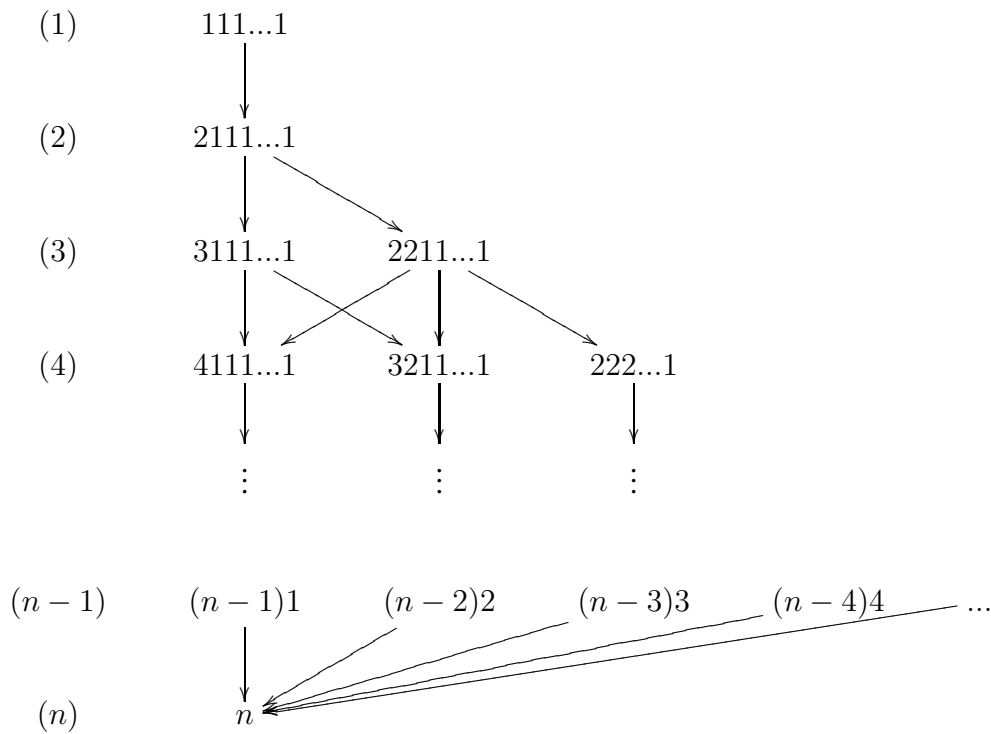
\begin{displaymath}
   \xymatrix{
       (1)& 111...1 \ar[d]   \\
        (2)& 2111...1  \ar[d] \ar[dr] \\
        (3)&3111...1  \ar[d] \ar[dr]&2211...1 \ar[dl] \ar[d] \ar[dr] \\
        (4)&4111...1 \ar[d] & 3211...1 \ar[d] &222...1\ar[d] \\
          &\vdots & \vdots & \vdots \\
(n-1)& (n-1)1 \ar[d] & (n-2)2 \ar[dl] & (n-3)3 \ar[dll] & (n-4)4 \ar[dlll] & ... \ar[dllll]  \\
(n)& n
          }
\end{displaymath}
It means that one integer partition can be transformed to another one by only using one transposition when there is an arrow connect two different integer partitions.

Define the graph above as $\Gamma(V,E,f)$, where $V$ is the set of all possible cycle types of the integer partitions for n. E is the set of pairs of vertices which are connected by arrows. In other words, $(a,b) \in E$ if and only if $a$ can change to $b$ (or the other way around) by a transposition. $f(a,b)$ equals to the number of ways that $a$ can transfer to $b$.

A direct result from this graph is that the graph of the different cycle types of an integer is \textit{bipartite}. As shown in the graph, for every $a,b \in V$, if $(a,b) \in E$, then their levels are adjacent. Therefore, $V=V_1\cup V_2$, where $V_1=\{a\in V| $ the level of $a$ is an odd number $\}$ and $V_2=\{a\in V| $ the level of $a$ is an even number $\}$. Therefore, if $(a,b)\in E$, then $a\in V_1 $ and $B\in V_2$ or the other way around.

By the previous proposition, $|V_1|\neq |V_2|$. By proposition 3.3 (b), $det(A_n)=0$
\end{proof}

\begin{remark}
The method of Proposition~\ref{Prop:bipartite-zero} can in fact be used to prove more. In the notation of the proof of part (b) of that proposition, say $X$ is $m \times k$ and $Y$ is $k \times m$, with $m > k$. Then using the Schur complement formula, the characteristic polynomial of $A_\Gamma$ is
\[
\begin{split}
\det (\lambda I_{m+k} - A_\gamma)
& = \det (\lambda I_m) \ \det(\lambda I_k - X (\lambda I_m)^{-1} Y)
= \lambda^m \det(\lambda I_k - X \lambda^{-1} Y) \\
& = \lambda^{m-k} \det(\lambda^2 I_k - X Y).
\end{split}
\]
Therefore $A_\Gamma$ has $0$ as an eigenvalue of multiplicity at least $\abs{m-k}$. For the partition matrix $A_n$, using Proposition~\ref{Prop:Aigner}, it follows that the multiplicity of the zero eigenvalue is at least the number of self-conjugate partitions of $n$. This fact will be re-proved using the explicit formula for the eigenvalue in Proposition~\ref{Prop:Rho-row-column}.
\end{remark}

\begin{proposition}
Let $\Gamma$ be a bipartite graph. Suppose $A_\Gamma$ has an eigenvector $\binom{u}{v}$ (in its standard decomposition) with eigenvalue $\rho$. Then $\binom{-u}{v}$ is also an eigenvector, with eigenvalue $-\rho$.
\end{proposition}

\begin{proof}
If
\[
\begin{pmatrix}
0 & X \\
Y & 0
\end{pmatrix}
\begin{pmatrix}
u \\ v
\end{pmatrix}
= \rho
\begin{pmatrix}
u \\ v
\end{pmatrix},
\]
then $X v = \rho u$, $Y u = \rho v$, and so
\[
\begin{pmatrix}
0 & X \\
Y & 0
\end{pmatrix}
\begin{pmatrix}
-u \\ v
\end{pmatrix}
= -\rho
\begin{pmatrix}
-u \\ v
\end{pmatrix}.
\]
\end{proof}

\begin{proof}[First proof of Proposition~\ref{Proposition:A-symmetries} parts (b,c)]
This follows from the preceding proposition and the observation that the partition graph is bipartite.
\end{proof}

%%%%%%%%%%%%%%%%%%%%%%%%%%%%%%%%%%%%%%%%%%%%%%%%%%%%%%%%%%%%%%%%%%%%%%%%%
\subsection{Border Strip Tableaux}\label{Section:Border-strip}%%%%%%%%%%%%%%%%%%%%%%%%%%%%%%%%%%%%
%raggedbottom%%%%%%%%%%%%%%%%%%%%%%%%%%%%%%%%%%%%%%%%%%%%%%%%%%%%%%%%%%%%

A border strip tableau of shape $\lambda$ and type $\mu$ is a Young diagram of shape $\lambda$ where the boxes are filled with numbers corresponding to $\mu$. If $\mu = \mu_1 \mu_2 ... \mu_n$, then we fill the boxes with $\mu_1$ 1s, $\mu_2$ 2s, etc, and $\mu_n$ n's. We must also follow the additional rules that, within any row or column, the numbers in the cells must be monotone increasing, and there must not be any $2 \times 2$ blocks of only one number (equivalently, the shape formed by any one number must be that of a staircase that only goes up and to the right). See Section 7.17 in \cite{Stanley-volume-2} for more details.

The height of a border strip is the number of rows it touches minus 1, and the height of a border strip tableau is the sum of the heights of its border strip (here, a border strip is the shape made by the cells of one number). We then have from the Murnaghan-Nakayama rule that

\begin{equation}
\label{Eq:Murnaghan}
\chi^\lambda(\mu) = \sum_{T} (-1)^{ht(T)}.
\end{equation}

%I have certain conjectures about these things that need to be proven. First, from Madeline's program, we have seen that eigenvalues come in positive and negative pairs. For n=3 and n=4 with the border strip tableau approach, I have seen that given two conjugate partitions, we generate the eigenvectors corresponding to a positive and a negative eigenvalue of the same magnitude. From this, it follows that a self-conjugate partition should generate an eigenvector for eigenvalue 0, and this has also been seen so far. We have also checked with Madeline's program to see that, so far, the multiplicity of 0 eigenvalues matches up with the number of self conjugate partitions of a given n.

We will prove that, for any border strip tableau, if its type is odd, then transposing the border strip changes the parity of its height (but not the heights of all of its parts obviously, just the sum of heights of the parts), and if the type is even, then transposing the border strip tableau does not change the parity of its height.

%We also conjecture that the series of hook type partitions are the partitions that generate the eigenvectors corresponding to the "n choose 2 minus multiples of n" series of eigenvalues. Perhaps this will be seen once we work on the "Jack Polynomials" idea that was offered as something to think about.

%(It was mentioned that the existence of the "n choose 2 minus multiples of n" eigenvalues was known from elsewhere.)

%This would also imply that if we draw the graph of partitions and label the vertices by the eigenvalues they generate, then there would be a chain corresponding to the hooks that is monotone. Could it be seen later than every chain is monotone?

%One idea I had is that the "empty" border strip has the same height and width of a border strip of only one cell. Now, given a non-empty border strip, if we add one more block to it, that will increase either the height or width (i.e. the height of the transpose), but not both. This is motivation for why it is the number of border strips that comprise the tableau (equivalent to the evenness or oddness of the type) that determines how the parities of height and the width compare.

%Now, our statement that it was the odd partitions that caused a sign change and the even ones that didn't was arbitrary since we are dealing with eigenvectors, so we can multiply by -1 as we please. In any case, the above formula shows that, for any parity of cells, there are is a parity of parts that forces height and width to have the same parity, and the other parity of parts forces height and width to have opposite parities.

\begin{proof}[First proof of Theorem~\ref{theorem;Main} part (c)]
Similarly to the height, we can define the width $wd(T)$ of a border strip tableau $T$ by $wd(T) = ht(T')$. In other words, we count the number of columns touched by each border strip, subtract 1, and then sum over all the border strips comprising the tableau. Then, for any border strip tableau $T$, we have the following identity:
\[
\text{height + width + parts = cells},
\]
that is,
\[
ht(T) + wd(T) + m = n,
\]
where $n$ is the total number of cells in $T$ and $m$ is the number of parts of $T$.

This is because, given any $T$, adding a new cell will increase $n$ by one and will increase either $ht(T)$ or $wd(T)$ by one if we add a new cell whose entry already exists in $T$ or will increase $m$ by one if the entry of this new cell doesn’t already exist in $T$.

From the identity $ht(T) + wd(T) + m = n$, alternatively, $ht(T) + ht(T’) = n-m$, we see that $ht(T)$ and $ht(T’)$ have the same parity if $n-m$ is even, and they have opposite parities if $n-m$ is odd.

Therefore, we have
\[
\chi^{\lambda'}(\mu) = \sum_T(-1)^{ht(T)+m-n}.
\]
If $T$ has shape $\lambda$, then $m = \ell(\lambda)$, so we get that
\[
\chi^{\lambda'}(\mu) = (-1)^{n - \ell(\lambda)} \chi^\lambda(\mu). \qedhere
\]
\end{proof}

%%%%%%%%%%%%%%%%%%%%%%%%%%%%%%%%%%%%%%%%%%%%%%%%%%%%%%%%%%%
%\section{Conclusions and Results}\label{s:intro}%%%%%%%%%%%%%%%%%%%%%%%%%%%%%%%%%%%%
%raggedbottom%%%%%%%%%%%%%%%%%%%%%%%%%%%%%%%%%%%%%%%%%%%%%%%%%%%%%%%%%%%%

%%%%%%%%%%%%%%%%%%%%%%%%%%%%%%%%%%%%%%%%%%%%
\subsection{Symmetric polynomials and the differential operator}
%%%%%%%%%%%%%%%%%%%%%%%%%%%%%%%%%%%%%%%%%%%%

\begin{definition}
Denote by $SP(N)$ the symmetric polynomials in $N$ variables. One basis for this vector space are the \emph{power sum symmetric polynomials}, defined by
\[
p_k(x_1, \ldots, x_N) = \sum_{i=1}^N x_i^k
\]
and, for a partition $\lambda = 1^{k_1} 2^{k_2} \ldots$,
\begin{equation}
\label{P-S-expansion}
p_\lambda = p_1^{k_1} p_2^{k_2} \ldots.
\end{equation}
Another basis for the same space are the Schur polynomials $\{s_\lambda\}$, also indexed by partitions. According to Corollary 7.17.5 in \cite{Stanley-volume-2},
\begin{equation}
\label{S-P-expansion}
s_\lambda = \sum_\nu \frac{1}{z_\nu} \chi^\lambda(\nu) p_\nu.
\end{equation}
This can be taken as the definition of $\chi^\lambda(\nu)$, although these can also be computed using character values, or tableaux as in an Section~\ref{Section:Border-strip}. Incidentally, according to Corollary 7.17.4,
\[
p_\nu = \sum_\lambda \chi^\lambda(\nu) s_\lambda.
\]
\end{definition}

The following result is, for example, Definition~2.10 in \cite{Dumitriu-MOPS}.

\begin{theorem}[Macdonald]
\label{theorem:Schur-definition}
The Schur polynomial $s_\lambda(x_1, \ldots, x_N)$ is an eigenfunction of the operator
\[
D^\ast = \sum_{i=1}^N x_i^2 \frac{\partial^2}{\partial x_i^2} + \sum_{i \neq j} \frac{1}{x_i - x_j} \left( x_i^2 \frac{\partial}{\partial x_i} - x_j^2 \frac{\partial}{\partial x_j} \right).
\]
with eigenvalue
\[
2 n (N-1) + 2 \rho_\lambda.
\]
\end{theorem}

\begin{theorem}
The matrix of the operator $\frac{1}{2} D^\ast$ with respect to the basis $\set{p_\lambda}_{\lambda \in P(n)}$ is, for $N > n$,
\[
A_n^T + n (N - 1),
\]
where $A_n$ is the matrix of the partition graph.
\end{theorem}

\begin{proof}
For a partition $\lambda = 1^{k_1} 2^{k_2} \ldots m^{k_m}$,
\[
\begin{split}
D^\ast(p_\lambda)
& = D^* \left(\prod_{l=1}^{m}p_l^{k_l} \right) \\
& = \sum_{i=1}^{N}x_{i}^2\frac{\partial }{\partial x_i^{2}}\prod_{l=1}^{m}p_l^{k_l}
+ \sum_{1\leq i\neq j\leq N}^{ }\frac{1}{x_i-x_j} \left(x_i^{2}\frac{\partial }{\partial x_i}(\prod_{l=1}^{m}p_l^{k_l})-x_j^{2}\frac{\partial }{\partial x_j}(\prod_{l=1}^{m}p_l^{k_l}) \right)
\end{split}
\]
\[
= \sum_{i=1}^{N}x_{i}^2\frac{\partial }{\partial x_i} \left((\prod_{l=1}^{m}p_l^{k_l})\sum_{l=1}^{m}\frac{k_llx_i^{l-1}}{p_l} \right)
+ \sum_{1\leq i\neq j\leq N}^{ }\frac{1}{x_i-x_j} \left((\prod_{l=1}^{m}p_l^{k_l})\sum_{l=1}^{m}\frac{k_ll(x_i^{l+1}-x_j^{l+1})}{p_l} \right)
\]
Using the product rule and the quotient rule in the first term, and algebra in the second, we get
\[
\begin{split}
& =(\prod_{l=1}^{m}p_l^{k_l}) \Biggl[ \sum_{i=1}^{N}x_{i}^2 \left(\sum_{l=1}^{m} \frac{k_l l ((l-1)p_lx_i^{l-2}-lx_i^{2l-2})}{p_l^{2}} + \left(\sum_{l=1}^{m}\frac{k_llx_i^{l-1}}{p_l} \right)^2 \right) \\
&\quad +\sum_{1\leq i\neq j\leq N}^{ } \sum_{l=1}^{m}\frac{k_ll}{p_l}\sum_{a=0}^{l} x_i^ax_j^{l-a} \Biggr]
\end{split}
\]
%\[
%\begin{split}
%& =(\prod_{l=1}^{m}p_l^{k_l})( \sum_{i=1}^{N}x_{i}^2(\sum_{l=1}^{m}(\frac{k_ll(l-1)x_i^{l-2}}{p_l}-\frac{lx_i^{2l-2}}{p_l^{2}})+(\sum_{l=1}^{m}(\frac{k_llx_i^{l-1}}{p_l}))^2) \\
%&\quad +\sum_{1\leq i\neq j\leq N}^{ }(\sum_{l=1}^{m}(\frac{k_ll}{p_l}\sum_{a=0}^{l}(x_i^ax_j^{l-a}))))
%\end{split}
%\]
\[
\begin{split}
& =(\prod_{l=1}^{m}p_l^{k_l})
\Biggl[ \sum_{i=1}^{N} \left(\sum_{l=1}^{m} \left(\frac{k_ll(l-1)x_i^{l}}{p_l}-\frac{k_ll^2x_i^{2l}}{p_l^{2}} \right) + \sum_{l=1}^{m}\frac{k_l^{2}l^{2}x_i^{2l}}{p_l^{2}} + \sum_{1\leq l_1\neq l_2\leq m}^{ }\frac{k_{l_1}k_{l_2}l_1l_2x_i^{l_1+l_2}} {p_{l_1}p_{l_2}} \right) \\
&\quad +\sum_{1\leq i\neq j\leq N}^{ } \sum_{l=1}^{m}\frac{k_ll}{p_l}\sum_{a=0}^{l} x_i^ax_j^{l-a} \Biggr]
\end{split}
\]
\[
\begin{split}
& =(\prod_{l=1}^{m}p_l^{k_l})
\Biggl[ \sum_{i=1}^{N} \left(\sum_{l=1}^{m} \left(\frac{k_ll(l-1)x_i^{l}}{p_l} + \frac{k_l (k_l - 1) l^2x_i^{2l}}{p_l^{2}} \right) + \sum_{1\leq l_1\neq l_2\leq m}^{ }\frac{k_{l_1}k_{l_2}l_1l_2x_i^{l_1+l_2}} {p_{l_1}p_{l_2}} \right) \\
&\quad +\sum_{1\leq i\neq j\leq N}^{ } \sum_{l=1}^{m}\frac{k_ll}{p_l}\sum_{a=0}^{l} x_i^ax_j^{l-a} \Biggr]
\end{split}
\]
Since $\sum_{i=1}^N x_i^t = p_t$ and $\sum_{i=1}^N x_i^0 = N$,
\[
\begin{split}
& = (\prod_{l=1}^{m}p_l^{k_l}) \Biggl[ \sum_{l=1}^{m} \left(\frac{k_ll(l-1)p_l}{p_l}+\frac{k_l(k_l-1)l^2p_{2l}}{p_l^{2}} \right)+\sum_{1\leq l_1\neq l_2\leq m}^{ }\frac{k_{l_1}k_{l_2}l_1l_2p_{l_1+l_2}} {p_{l_1}p_{l_2}}  \\
&\quad +\sum_{l=1}^{m}\frac{k_ll}{p_l} \left(\sum_{a=1}^{l-1}(p_ap_{l-a}-p_l) + 2 p_l (N-1) \right) \Biggr]
\end{split}
\]
\[
\begin{split}
& =(\prod_{l=1}^{m}p_l^{k_l}) \Biggl[ \sum_{l=1}^{m} \left(k_ll(l-1)+\frac{k_l(k_l-1)l^2p_{2l}}{p_l^{2}} \right)+\sum_{1\leq l_1\neq l_2\leq m}^{ }\frac{k_{l_1}k_{l_2}l_1l_2p_{l_1+l_2}} {p_{l_1}p_{l_2}}  \\
&\quad + \sum_{l=1}^{m} k_ll \left( \sum_{a=1}^{l-1}\left(\frac{p_ap_{l-a}}{p_l}-1 \right)+2(N-1) \right) \Biggr]
\end{split}
\]
%\[
%\begin{split}
%& =(\prod_{l=1}^{m}p_l^{k_l})( (\sum_{l=1}^{m}(k_ll(l-1)+\frac{k_l(k_l-1)l^2p_{2l}}{p_l^{2}})+\sum_{1\leq l_1\neq l_2\leq m}^{ }\frac{k_{l_1}k_{l_2}l_1l_2p_{l_1+l_2}} {p_{l_1}p_{l_2}} ) \\
%&\quad +(\sum_{l=1}^{m}(k_ll(\sum_{a=1}^{l-1}(\frac{p_ap_{l-a}}{p_l}-1)+2(n-1)))))
%\end{split}
%\]

\[
\begin{split}
& =(\prod_{l=1}^{m}p_l^{k_l}) \Biggl[ \sum_{l=1}^{m}\frac{k_l(k_l-1)l^2p_{2l}}{p_l^{2}} + \sum_{1\leq l_1\neq l_2\leq m}^{ }\frac{k_{l_1}k_{l_2}l_1l_2p_{l_1+l_2}} {p_{l_1}p_{l_2}}  \\
&\quad +\sum_{l=1}^{m}\left(k_ll \left(\sum_{a=1}^{l-1}\frac{p_ap_{l-a}}{p_l}+2(N-1) \right) \right) \Biggr]
\end{split}
\]
%\[
%=(\prod_{l=1}^{m}p_l^{k_l})[2n(n-1)+\sum_{l=1}^{m}k_ll(\frac{(k_l-1)lp_{2l}}{p_l^{2}} + \sum_{a=1}^{l-1}\frac{p_ap_{l-a}}{p_l} )+\sum_{1\leq l_1\neq l_2\leq m}^{ }\frac{k_{l_1}k_{l_2}l_1l_2p_{l_1+l_2}} {p_{l_1}p_{l_2}}]
%\]

%$$D^{*}\bigg( \prod_{l=1}^m p_{l}^{k_l} \bigg) =  \prod_{l=1}^m p_{l}^{k_l}  $$

Since $\sum_{l=1}^m k_l l = n$, we conclude that
\begin{multline*}
D^{*}\bigg( \prod_{l=1}^m p_{l}^{k_l} \bigg)  \\
=  \prod_{l=1}^m p_{l}^{k_l} \left[ 2(N-1)n +	\sum_{l=1}^{m} k_l l \bigg( \frac{(k_l - 1) l p_{2l}}{p_l^2} + \sum_{a=1}^{l-1}\frac{p_a p_{l-a}}{p_l}\bigg) + \sum_{1 \leq l_1 \neq l_2 \leq m} \frac{k_{l_1}k_{l_2}l_1 l_2 (p_{l_1+l_2})}{p_{l_1} p_{l_2}} \right].
\end{multline*}
Comparing coefficients with Proposition~\ref{Proposition:Transition-matrix}, we obtain the result.
\end{proof}

\begin{corollary}
\label{corollary:A-transpose}
For each $\lambda$, the vector $(\frac{1}{z_\nu} \chi^\lambda(\nu))_{\nu \in P(n)}$ is an eigenvector of the matrix $A_n^T$ with eigenvalue $\rho_\lambda$.
\end{corollary}

\begin{proof}
This follows from Theorem~\ref{theorem:Schur-definition} and the expansion \eqref{S-P-expansion}.
\end{proof}

\begin{lemma}
\label{Lemma:Dual-bases}
The bases
\[
\set{ u_\lambda = (\chi^\lambda(\nu))_{\nu \in P(n)} : \lambda \in P(n)} \quad \text{and} \quad \set{\left(\frac{1}{z_\nu} \chi^\lambda(\nu)\right)_{\nu \in P(n)} : \lambda \in P(n)}
\]
are dual to each other.
\end{lemma}

\begin{proof}
Combining Corollaries 7.17.4 and 7.17.5 in \cite{Stanley-volume-2},
\[
s_\lambda = \sum_\nu \frac{1}{z_\nu} \chi^\lambda(\nu) \sum_\mu \chi^\mu(\nu) s_\mu
= \sum_\mu \left( \sum_\nu \frac{1}{z_\nu} \chi^\lambda(\nu) \chi^\mu(\nu) \right) s_\mu,
\]
and so
\[
(\chi^\lambda(\nu))_{\nu \in P(n)} \cdot \left(\frac{1}{z_\nu} \chi^\mu(\nu)\right)_{\nu \in P(n)}
= \sum_{\nu \in P(n)} \frac{1}{z_\nu} \chi^\lambda(\nu) \chi^\mu(\nu) = \delta_{\lambda \mu}. \qedhere
\]
\end{proof}

\begin{lemma}
\label{Lemma:Transpose}
Let $\set{e_i}$ be a basis of eigenvectors of matrix $C$, with eigenvalues $\set{\rho_i}$. Let $\set{f_i}$ be the dual basis, determined by the relation
\[
e_i \cdot f_j = \delta_{i j}.
\]
Then $\set{f_i}$ is a basis of eigenvectors for $C^T$, with eigenvalues $\set{\rho_i}$.
\end{lemma}

\begin{proof}
For all $i$,
\[
e_i \cdot C^T f_j = C e_i \cdot f_j = \rho_i e_i \cdot f_j = \rho_i \delta_{i j} = \rho_j \delta_{i j} = e_i \cdot (\rho_j f_j).
\]
Therefore $C_T f_j = \rho_j f_j$.
\end{proof}

\begin{proof}[Proof of Proposition~\ref{Proposition:Eigenvectors-A}]
The result follows by combining Corollary~\ref{corollary:A-transpose} with Lemmas~\ref{Lemma:Dual-bases} and \ref{Lemma:Transpose}.
\end{proof}

\begin{proof}[Proof of Theorem~\ref{theorem;Main} part (a)]
It follows from Lemma~\ref{Lemma:Dual-bases} that
\[
\begin{pmatrix} 1 \\ 0 \\ \vdots \\ 0 \end{pmatrix} = \sum_\lambda \frac{1}{z_{1^n}} \chi^\lambda(1^n) u_\lambda.
\]
Therefore
\[
\begin{split}
c_k(\mu)
& = \left( A_n^k \begin{pmatrix} 1 \\ 0 \\ \vdots \\ 0 \end{pmatrix} \right)_\mu
= \left(A_n^k \frac{1}{z_{1^n}} \sum_\lambda \chi^\lambda(1^n) u_\lambda \right)_\mu \\
& = \left(\frac{1}{z_{1^n}} \sum_\lambda \rho_\lambda^k \chi^\lambda(1^n) u_\lambda \right)_\mu
= \frac{1}{z_{1^n}} \sum_\lambda \rho_\lambda^k \chi^\lambda(1^n) \chi^\lambda(\mu).
\end{split}
\]
It remains to note that $z_{1^n} = n!$.
\end{proof}

The following is Theorem 7.14.5 in \cite{Stanley-volume-2}.

\begin{proposition}
Define an involution $\omega$ on $SP(N)$ by
\[
\omega p_\lambda = \epsilon_\lambda p_\lambda,
\]
where $\epsilon_\lambda = (-1)^{n - \ell(\lambda)}$ and $\ell(\lambda)$ is the number of parts (rows) in $\lambda$. Then also
\[
\omega s_\lambda = s_{\lambda'},
\]
where $\lambda'$ is the conjugate partition to $\lambda$, the one obtained by flipping rows and columns.
\end{proposition}

\begin{proof}[Second proof of Theorem~\ref{theorem;Main} part (c)]
Using the preceding proposition,
\[
\sum_\mu \frac{1}{z_\mu} \chi^{\lambda'}(\mu) p_\mu
= s_{\lambda'}
= \omega s_{\lambda}
= \sum_\mu \frac{1}{z_\mu} \chi^{\lambda}(\mu) \omega p_\mu
= \sum_\mu \frac{1}{z_\mu} (-1)^{n - \ell(\mu)} \chi^{\lambda}(\mu) p_\mu \qedhere
\]
\end{proof}

\begin{proposition}
\label{Prop:Rho-row-column}
$\rho_{\lambda}=\sum_{b \in \lambda }^{ }(b_x-b_y)$

where $b$ is a box in the Young diagram for $\lambda$ and $b_x$ and $b_y$ are the column and row indices of $b$ respectively.
\end{proposition}

\begin{proof}
The formula for an eigenvalue corresponding to a certain partition $\lambda$ is
\[
2 \rho_\lambda=\sum_{i=1}^{n} \lambda_i( \lambda_i-2i+1).
\]
Let $\lambda^\ast$ be another partition whose Young diagram has one more block than the Young diagram for $\lambda$. This means that for some value of $i$ between $1$ and $\ell(\lambda)$, $\lambda_i$ is increased by one (and we also permit increasing $\lambda_{\ell(\lambda) + 1}$ from $0$ to $1$ meaning the row was initially of length $0$). Then
\[
2 \rho_{\lambda^\ast} - 2 \rho_{\lambda}=(\lambda_i+1)(\lambda_i-2i+2) -\lambda_i(\lambda_i-2i+1).
\]
Expanding and canceling like terms yields
\[
\rho_{\lambda^\ast} - \rho_{\lambda}=\lambda_i+1-i.
\]
If we were to index the rows and columns of Young diagrams starting with 1 from the top and the left respectively, we see that $\lambda_1+1$ is the column index and $i$ is the row index of the one block that we added to $\lambda$ to form $\lambda^\ast$. Thus, since it is clear that
\[
\rho_{\lambda_0} = 0,
\]
where $\lambda_0$ is the empty partition (has an empty Young diagram), we get that
\[
\rho_{\lambda}=\sum_{b \in \lambda }^{ }(b_x-b_y),
\]
where $b$ is a block of the Young diagram for $\lambda$ and $b_x$ and $b_y$ are the column and row indices of $b$ respectively.
\end{proof}

\begin{proof}[Proof of Theorem~\ref{theorem;Main} part (b)]

We now see that replacing $\lambda$ with $\lambda'$ switches $b_x$ and $b_y$ for each block in $\lambda$, so we get the desired result that
$\rho_{\lambda} = -\rho_{\lambda'}$
In particular, if $\lambda$ is self-conjugate, then
$\rho_{\lambda} = 0$.
\end{proof}

\begin{proof}[Second proof of Proposition~\ref{Proposition:A-symmetries}]
Combine Theorem~\ref{theorem;Main} parts (b,c) with Proposition~\ref{Proposition:Eigenvectors-A}.
\end{proof}

%\begin{proposition}
%Let $\lambda$ be a partition and $\lambda'$ the conjugate partition. Then
%\[
%\sum_i \lambda_i (i-1) = \sum_j \binom{\lambda_j'}{2}.
%\]
%\end{proposition}
%
%I expect that this proposition can be proved by looking at the Young diagram of $\lambda$ and computing one quantity in two different ways.
%
%\begin{corollary}
%\[
%\sum_{i=1}^N \lambda_i (\lambda_i - 1 - \frac{2}{\alpha} (i-1))
%= \frac{2}{\alpha} \sum_{i=1}^N \lambda_i (\frac{\alpha}{2} (\lambda_i - 1) - (i-1))
%= \frac{2}{\alpha} \left( \alpha \sum_{i=1}^N \binom{\lambda_i}{2} - \sum_{i=1}^N \lambda_i (i-1) \right)
%\]
%for $\alpha = 1$ changes sign when $\lambda$ is replaced by $\lambda'$.
%\end{corollary}

\subsection{Examples}

\begin{lemma}
If a permutation $\alpha$ has $k$ cycles, so that the corresponding partition $\mu$ has $k$ parts, then the sum in Theorem~\ref{theorem;Main} needs to be taken only over partitions $\lambda$ whose Young diagram can be represented as a union of $k$ hooks, that is, over diagrams which do not contain a $(k+1)$-by-$(k+1)$ square.
\end{lemma}

\begin{proof}
This follows from the description of the coefficients $\chi^\lambda(\mu)$ in the expansion in terms of border strip tableaux of shape $\lambda$ and type $\mu$.
\end{proof}

The case of a single cycle is treated in the Introduction. We now describe the case of two cycles.

\begin{lemma}
Denote
\[
\lambda_{a,b} = a 1^b \in P(n), \qquad 0 < a, \quad 0 \leq b, \quad a+b=n,
\]
the hooks, and
\[
\lambda_{a, b, c, d} = a (c+1) 2^d 1^{b-d-1} \in P(n), \qquad 1 \leq c < a, \quad 0 \leq d < b, \quad a + b + c + d = n,
\]
general partitions whose Young diagrams contain a 2-by-2 square but not a 3-by-3 square. Then
\[
\chi^{\lambda_{a,b,c,d}}(1^n) = \binom{n}{a, b, c, d} \frac{a c (a-c) (b-d)}{(a+b) (a+d) (b+c) (c+d)}
\]
and
\[
\rho_{\lambda_{a,b,c,d}} = \binom{a}{2} - \binom{b+1}{2} + \binom{c}{2} - \binom{d+1}{2}.
\]
\end{lemma}

\begin{proof}
The product of hook lengths in the diagram $\lambda_{a,b,c,d}$ is
\[
\begin{split}
&(a+b) (a+d) (b+c) (c+d) \prod_{i=1}^{a-c-1} i \prod_{i=a-c}^{a-2} (i+1) \prod_{j=1}^{c-1} j \prod_{i=1}^{b-d-1} i \prod_{i=b-d}^{b-1} (i+1) \prod_{j=1}^d j \\
& = (a+b) (a+d) (b+c) (c+d) \frac{(a-1)!}{a-c} (c-1)! \frac{b!}{b-d} d!.
\end{split}
\]
So using the hook length formula, Corollary 7.21.6 in \cite{Stanley-volume-2},
\[
\chi^{\lambda_{a,b,c,d}}(1^n) = \binom{n}{a, b, c, d} \frac{a c (a-c) (b-d)}{(a+b) (a+d) (b+c) (c+d)}
\]

The formula for $\rho_{\lambda_{a,b,c,d}}$ follows from Proposition~\ref{Prop:Rho-row-column}.
\end{proof}

\begin{corollary}
Let $\mu$ have two parts of sizes $m \geq k$, so that the corresponding permutation has two cycles of these sizes. The only partitions which contribute to the sum in Theorem~\ref{theorem;Main} are of the form $\lambda_{a,b}$ or $\lambda_{a,b,c,d}$ from the preceding lemma. More explicitly, denote
\[
\Lambda_1(m,k) = \set{\lambda_{i, m-i, j, k-j} : 1 \leq j \leq k, j < i < m - k + j},
\]
\[
\Lambda_2(m,k) = \set{\lambda_{i, k-j, j, m-i} : 1 \leq j \leq k, m - k + j < i \leq m},
\]
\[
\Lambda_3(m,k) = \set{\lambda_{i, m-j, j, k-i} : 1 \leq j < i \leq k},
\]
\[
\Lambda_4(m,k) = \set{\lambda_{i+k, m-i} : m-k < i \leq m},
\]
\[
\Lambda_5(m,k) = \set{\lambda_{i, m-i+k} : 1 \leq i < k}.
\]
\[
\Lambda_6(m,k) = \set{\lambda_{i+k, m-i} : 0 \leq i \leq m-k} = \set{\lambda_{i, m-i+k} : k \leq i \leq m},
\]
For $m > k$, all of these sets are disjoint, each partition on $\Lambda_1 \cup \ldots \cup \Lambda_5$ has a single border strip tableau with filling $\mu$, and each partition in $\Lambda_6$ has two such tableaux. If $m = k$, then $\Lambda_1(k,k) = \emptyset$, $\Lambda_6(k,k) = \set{\lambda_{k, k}}$, and $\Lambda_2(k,k) = \Lambda_3(k,k)$ with each partition in the latter set having two border strip tableaux with filling $\mu$.

Then
\[
c_k(\mu) = \frac{1}{n!} \sum_{\lambda \in \Lambda} \chi_{\lambda}(\mu) \chi^\lambda(1^n) \rho_\lambda^k.
\]
Here $\chi^\lambda(1^n)$ and $\rho_\lambda$ are as in the preceding lemma. For $m > k$,
\[
\Lambda = \Lambda_1 \cup \Lambda_2 \cup \Lambda_3 \cup \Lambda_4 \cup \Lambda_5 \cup \Lambda_6,
\]
and
\[
\chi^{\lambda_{a,b,c,d}}(\mu) =
\begin{cases}
(-1)^{b+d}, & \lambda_{a,b,c,d} \in \Lambda_1(m,k), \\
(-1)^{b+d+1}, & \lambda_{a,b,c,d} \in \Lambda_2(m,k) \cup \Lambda_3(m,k).
\end{cases}
\]
For $m=k$,
\[
\Lambda = \Lambda_2 \cup \Lambda_4 \cup \Lambda_5 \cup \Lambda_6,
\]
and
\[
\chi^{\lambda_{a,b,c,d}}(\mu) =
2 (-1)^{b+d+1}, \quad \lambda_{a,b,c,d} \in \Lambda_2(k,k).
\]
In either case,
\[
\chi^{\lambda_{a,b}}(\mu) =
(-1)^{b}, \quad \lambda_{a,b} \in \Lambda_4(m,k) \cup \Lambda_5(m,k)
\]
and
\[
\chi^{\lambda_{a,b}}(\mu) =
\begin{cases}
2 (-1)^{b}, & \lambda_{a,b} \in \Lambda_6(m,k), k \text{ odd}, \\
0, & \lambda_{a,b} \in \Lambda_6(m,k), k \text{ even}.
\end{cases}
\]
\end{corollary}

\begin{proof}
We only sketch the argument; the details are left to the interested reader. The partitions $\lambda$ which are not hooks and have non-zero $\chi^{\lambda}(\mu)$ have the form
\[
\lambda_{\max(i,j), \max(m-i, k-j), \min(i,j), \min(m-i, k-j)}
\]
for $1 \leq i < m$, $1 \leq j \leq k$, $i \neq j$, $m-i \neq k-j$. The sets $\Lambda_1-\Lambda_4$ parameterize this collection. The values of $\chi^{\lambda}(\mu)$ are obtained using the border strip tableaux representation.
\end{proof}

%\bibliographystyle{amsalpha}
%\bibliography{bibdata}

\def\cprime{$'$} \def\cprime{$'$}
\providecommand{\bysame}{\leavevmode\hbox to3em{\hrulefill}\thinspace}
\providecommand{\MR}{\relax\ifhmode\unskip\space\fi MR }
\providecommand{\MRhref}[2]{%
  \href{http://www.ams.org/mathscinet-getitem?mr=#1}{#2}
}
\providecommand{\href}[2]{#2}

\raggedright
%%%%%%%%%%%%%%%%%%%%%%%%%%%%%%%%%%%%%%%%%%%%%%%%%%%%%%%%%%%%%%%%%%%%%%%%%
%\begin{thebibliography}{HHHKR10}%%%%%%%%%%%%%%%%%%%%%%%%%%%%%%%%%%%%%%%%%
%\raggedbottom%%%%%%%%%%%%%%%%%%%%%%%%%%%%%%%%%%%%%%%%%%%%%%%%%%%%%%%%%%%%

%%%%%%%%%%%%%%%%%%%%%%%%%%%%%%%%%%%%%%%%%%%%%%%%%%%%%%%%%%%%%%%%%%%%%%%%%
%\end{thebibliography}%%%%%%%%%%%%%%%%%%%%%%%%%%%%%%%%%%%%%%%%%%%%%%%%%%%%
%%%%%%%%%%%%%%%%%%%%%%%%%%%%%%%%%%%%%%%%%%%%%%%%%%%%%%%%%%%%%%%%%%%%%%%%%

%%%%%%%%%%%%%%%%%%%%%%%%%%%%%%%%%%%%%%%%%%%%%%%%%%%%%%%%%%%%%%%%%%%%%%%%%
\end{document}